\documentclass[12pt]{amsart}

\usepackage[margin=1.0in]{geometry}
\usepackage{amsmath}
\usepackage{amssymb}
\usepackage[all]{xy}
\usepackage{tikz}
\usetikzlibrary{calc}

\newcommand{\rat}{\dashrightarrow}
\newcommand{\A}{\mathbb A}
\renewcommand{\P}{\mathbb P}
\newcommand{\C}{\mathbb C}
\newcommand{\R}{\mathbb R}
\newcommand{\Q}{\mathbb Q}
\newcommand{\cO}{\mathcal O}

\newcommand{\abs}[1]{\left\lvert #1 \right\rvert}

\newtheorem{theorem}{Theorem}
\newtheorem{corollary}[theorem]{Corollary}
\newtheorem{lemma}[theorem]{Lemma}
\newtheorem{proposition}[theorem]{Proposition}

\theoremstyle{definition}
\newtheorem{definition}{Definition}
\newtheorem{remark}{Remark}
\newtheorem{example}[theorem]{Example}

\DeclareMathOperator{\codim}{codim}

\DeclareMathOperator{\Bl}{Bl}
\DeclareMathOperator{\Amp}{Amp}
\DeclareMathOperator{\Mov}{Mov}
\DeclareMathOperator{\Movb}{\overline{Mov}}
\DeclareMathOperator{\Effb}{\overline{Eff}}
\DeclareMathOperator{\Nef}{Nef}
\DeclareMathOperator{\indet}{indet}
\DeclareMathOperator{\mult}{mult}
\DeclareMathOperator{\Spec}{Spec}

\begin{document}

\title[Relative asymptotic multiplicity]{A pathology of asymptotic multiplicity \\ in the relative setting}
\author{John Lesieutre} 
\address{Institute for Advanced Study\\
Einstein Dr\\ Princeton, NJ 08540, USA} \email{johnl@math.ias.edu}

\begin{abstract}
We point out an example of a projective family \(\pi : X \to S\), a \(\pi\)-pseudoeffective divisor \(D\) on \(X\), and a subvariety \(V \subset X\) for which the asymptotic multiplicity \(\sigma_V(D;X/S)\) is infinite. This shows that the divisorial Zariski decomposition is not always defined for pseudoeffective divisors in the relative setting.
\end{abstract}

\maketitle

\section{Introduction}

Suppose that \(X\) is a smooth projective variety and \(D\) is a pseudoeffective \(\R\)-divisor on \(X\).  The asymptotic multiplicity of \(D\) along a subvariety \(V \subset X\), studied by Nakayama~\cite{nakayama} and Ein-Lazarsfeld-Musta{\c{t}}{\u{a}-Nakamaye-Popa~\cite{elmnp}, has proved to be a fundamental tool in understanding the properties of the divisor \(D\).  For big divisors \(D\), the definition of the asymptotic multiplicity is straightforward: roughly, one considers the linear series \(\abs{mD}\) for larger and larger values of \(m\), and takes \(\sigma_V(D) = \lim_{m \to \infty} \frac{1}{m} \mult_V \abs{mD}\), where the multiplicity of a linear series along a subvariety is defined to be the multiplicity of a general member.

Complications arise, however, in carrying out this construction for divisors \(D\) which are pseudoeffective but not big, i.e.\ for divisors on the boundary of the pseudoeffective cone \(\Effb(X) \subset N^1(X)_\R\). Nakayama realized that \(\sigma_V(D)\) can be extended to a lower semicontinuous function on \(\Effb(X)\) by setting
\[
\sigma_V(D) = \lim_{\epsilon \to 0} \sigma_V(D+ \epsilon A),
\]
where \(A\) is a fixed ample divisor.  In some applications (e.g.\ in the construction of Zariski decompositions), it is important to know that the limit in question takes a finite value. While it is clear that the quantity on the right is nondecreasing as \(\epsilon\) is made smaller, it might \emph{a priori} be unbounded in the limit. That this does not happen in the non-relative setting was observed by Nakayama.

Our aim in this note is to demonstrate by example that when asymptotic multiplicity invariants are considered in the greater generality of divisors on a projective family \(\pi : X \to S\), this finiteness need not hold: for a \(\pi\)-pseudoeffective divisor, the limit defining \(\sigma_V(D;X/S)\) can indeed be infinite. This answers a question of Nakayama~\cite[pg.\ 33]{nakayama}.  The example itself is familiar, a divisor on the versal deformation space of a fiber of Kodaira type \(I_2\), which has been considered in related contexts by Reid~\cite[6.8]{reid} and Kawamata~\cite[Example 3.8(2)]{kawamatacy},\cite[Example 9]{kawamataremarks}.

\begin{theorem}
\label{cyrelativeex}
There exists a projective family \(\pi : X \to S\), a \(\pi\)-pseudoeffective divisor \(D\), and a subvariety \(V \subset X\) for which \(\sigma_V(D;X/S)\) is infinite.
\end{theorem}

An important use of asymptotic multiplicity invariants is in the construction of the divisorial Zariski decomposition, a higher-dimensional analog of the usual Zariski decomposition on surfaces.  The example here shows that trouble arises if one generalizes this construction to pseudoeffective classes in the relative setting: after passing to a blow-up on which the valuation corresponding to \(V\) is divisorial, we obtain an example in which the decomposition is not defined.

\begin{corollary}
\label{nodzd}
Let \(\pi : X \to S\) be as in Theorem~\ref{cyrelativeex}.  If \(f : W \to X\) is the blow-up along \(V\) with exceptional divisor \(E\), then \(\tilde{D} = f^\ast D\) has \(\sigma_E(\tilde{D}; W/S) = \infty\) and \(N_\sigma(\tilde{D};W/S)\) is not defined.
\end{corollary}

Moreover, the divisor \(\tilde{ D}\) does not admit any Zariski decomposition in a very strong sense:
\begin{corollary}
\label{nowzd}
There does not exist a birational model \(g : Z \to W\) for which \(g^\ast \tilde{D}\) admits a decomposition \(g^\ast \tilde{D} = P+N\) with \(P\) a \(g \circ (f \circ \pi)\)-movable divisor and \(N\) effective.
\end{corollary}

In Section~\ref{prelims} we recall the basic definitions and properties of the invariants \(\sigma_V(D;X/S)\) and \(N_\sigma(D;X/S)\) appearing in Theorem~\ref{cyrelativeex} and Corollary~\ref{nodzd}, before establishing the claims in Section~\ref{examplesect}.  In Section~\ref{generalsetup}, we describe a more general setting for making computations in a similar spirit.

\section{Preliminaries}
\label{prelims}

Suppose that \(\pi : X \to S\) is a projective, surjective morphism with connected fibers, with \(X\) and \(S\) normal and \(\Q\)-factorial (hereafter, a \emph{nice family}).  We will find it convenient to allow the base \(S\) to be a surface germ, following~\cite{kawamatacrepant}.  Two divisors \(D\) and \(D^\prime\) on \(X\) are said to be numerically equivalent over \(S\), or \(\pi\)-numerically equivalent, if \(D \cdot C = D^\prime \cdot C\) for any curve \(C\) that is contracted by \(\pi\); write \(D \equiv_\pi D^\prime\) for the relation of numerical equivalence over \(S\), and \(N^1(X/S)\) for the vector space of \(\R\)-divisors on \(X\), modulo this equivalence.

The familiar cones of positive divisors on a projective variety all have analogs in the relative setting: a divisor \(D\) on \(X\) is said to be
\begin{enumerate}
\item \(\pi\)-ample if \(D_s\) is ample on every fiber \(X_s = \pi^{-1}(s)\); \item \(\pi\)-nef if \(D_s\) is nef on every fiber \(X_s\) (i.e.\ if \(D \cdot C \geq 0\) for every curve \(C\) contracted by \(\pi\));
\item \(\pi\)-movable if the support of the cokernel of \(f^\ast f_\ast \cO_X(D) \to \cO_X(D)\) has codimension at least \(2\);
\item \(\pi\)-big if the restriction of \(D\) to the generic fiber is big;
\item \(\pi\)-pseudoeffective if the restriction of \(D\) to the generic fiber is pseudoeffective.
\end{enumerate}
Corresponding to these classes of divisors are cones inside \(N^1(X/S)\):
\[
\Amp(X/S) \subseteq \Mov(X/S) \subseteq \Effb(X/S).
\]
We note that the cone \(\Effb(X/S)\) is not necessarily a strongly convex cone, in that it might contain an entire line through the origin; this contrasts with the familiar case when \(S\) is a point. For example, if \(D\) restricts to \(0\) on a general fiber of \(\pi\), then \(D\) and \(-D\) are both \(\pi\)-pseudoeffective.

For simplicity, we will assume that the base space \(S\) is affine and that there exists a \(\pi\)-ample divisor \(A\) on \(X\). This is not really necessary, but the invariants under consideration can be computed in the general setting simply restricting to the preimage of a suitable affine open set; we refer to~\cite[\S 3.2]{nakayama} for details.  If \(D\) is a \(\pi\)-big divisor, then \(f_\ast \cO_X(mD) \neq 0\) for sufficiently large and divisible \(m\), and so \(H^0(X,\cO_X(mD)) = f_\ast(\cO_X(mD))\) is nonzero as well.  Hence if \(S\) is affine, any \(\pi\)-big class has an effective representative.

\begin{definition}
\label{sigmadef}
Given an irreducible subvariety \(V \subset X\) and a \(\pi\)-big \(\R\)-divisor \(D\), set
\[
\sigma_V(D; X/S) = \inf_{\substack{ D^\prime \equiv_\pi D \\ D^\prime \geq
0}} \mult_V(D^\prime).
\]
Since \(D\) is \(\pi\)-big, there exists an effective \(\R\)-divisor \(D^\prime\) that is \(\pi\)-numerically equivalent to \(D\), and this infimum is taken over a nonempty set.
\end{definition}

In the definition, \(D^\prime\) ranges over effective \(\R\)-divisors that are \(\pi\)-numerically equivalent to \(D\).  When \(S = \Spec \C\) and \(D\) is a big integral divisor, a sequence \(D^\prime_m\) of such \(\R\)-divisors with multiplicities converging to the infimum can be found by taking \(D^\prime_m \in \frac{1}{m} \abs{mD}\), where we choose a general element of the linear system \(\abs{mD}\).

We next extend the definition of the asymptotic multiplicity from \(\pi\)-big divisors to \(\pi\)-pseudoeffective divisors.

\begin{definition}
Given a \(\pi\)-pseudoeffective \(\R\)-divisor \(D\), set
\[
\sigma_V(D;X/S) = \lim_{\epsilon \to 0} \sigma_V(D + \epsilon A; X/S).
\]
\end{definition}

This is evidently a nondecreasing function as \(\epsilon\) approaches \(0\), but it might have infinite limit.  To show that it has a finite limit, it suffices to bound \(\sigma_V(D+\epsilon A;X/S)\) above, independent of \(\epsilon\). Nakayama gives several conditions under which this can be achieved.

\begin{theorem}[\cite{nakayama}, Lemmas 2.1.2, 3.2.6]
\label{finiteconds}
If any of the following holds, then \(\sigma_V(D;X/S)\) is finite.
\begin{enumerate}
\item \(S = \Spec \C\) is a point;
\item \(D\) is numerically equivalent over \(S\) to an effective \(\R\)-divisor \(\Delta\);
\item \(\codim \pi(V) < 2\).
\end{enumerate}
\end{theorem}
We recall the proof in case (1), perhaps the most important in practice.  Case (2) is immediate from the definition, and we refer to \cite{nakayama} for (3).  Assume for a moment that \(V \subset X\) is an irreducible divisor; that this implies the general statement will follow from Theorem~\ref{sigmaproperties}(2) below.
\begin{proof}[Proof of (1)]
For any \(\epsilon\), \((D+\epsilon A) - \sigma_V(D + \epsilon A) V\) is pseudoeffective, and so
\[
\left( (D+\epsilon A) -\sigma_V(D + \epsilon A) V \right) \cdot A^{n-1}  \geq 0.
\]
As long as \(\epsilon < 1\) it follows that
\[
\sigma_V(D+\epsilon A) \leq \frac{(D+\epsilon A) \cdot A^{n-1}}{V \cdot A^{n-1}} \leq \frac{(D+A) \cdot A^{n-1}}{V \cdot A^{n-1}}
\] 
is bounded above as \(\epsilon\) decreases to \(0\).
\end{proof}
This argument relies in a crucial way on the properness of \(X\) to carry out intersection theory, and is not applicable in the relative setting in general.  

\begin{proposition}[\cite{nakayama}, Lemmas 2.1.4, 2.2.2, 2.1.7]
\label{sigmaproperties}
Suppose that \(\pi : X \to S\) is a nice family and \(V \subset X\) is an irreducible subvariety.
\begin{enumerate}
\item If \(F\) is any \(\pi\)-pseudoeffective divisor on \(X\), then
\[
\lim_{\epsilon \to 0} \sigma_V(D+\epsilon F;X/S) = \sigma_V(D;X/S).
\]
\item Let \(f : W \to X\) be the normalized blow-up of \(X\) along \(V\), and let \(E\) be a component of the exceptional divisor over \(V\).  Then \(\sigma_E(f^\ast D;W/S) = \sigma_V(D;X/S)\).
\item The number of prime divisors \(\Gamma\) for which \(\sigma_\Gamma(D;X/S) > 0\) is finite.
\end{enumerate}
\end{proposition}

The first of these shows that Definition~\ref{sigmadef} is independent of the choice of \(\pi\)-ample divisor \(A\), while the second completes the proof of Theorem~\ref{sigmaproperties}.

\begin{definition}[\cite{nakayama}, \cite{boucksomzariski}]
Suppose that \(\pi : X \to S\) is a nice family and that \(D\) is a \(\pi\)-pseudoeffective divisor such that \(\sigma_{\Gamma}(D;X/S)\) is finite for every prime divisor \(\Gamma\).  Then set
\begin{align*}
N_\sigma(D;X/S) &= \sum_\Gamma \sigma_\Gamma(D;X/S) \, \Gamma, \\ P_\sigma(D;X/S) &= D - N_\sigma(D;X/S).
\end{align*}
It follows from Proposition~\ref{sigmaproperties}(3) that there are only finitely many nonzero terms in the sum defining \(N_\sigma(D;X/S)\).
\end{definition}

We refer to \(N_\sigma(D;X/S)\) as the negative part of the Zariski decomposition, and \(P_\sigma(D;X/S)\) as the positive part.  The negative part is a rigid, effective divisor.  The positive part might not be nef, but it lies in the closure \(\Movb(X/S)\) of the cone \(\Mov(X/S)\).  Corollary~\ref{nodzd} shows that without the finiteness hypothesis on \(\sigma_\Gamma(D;X/S)\), the definition is not always applicable in the relative setting.

An equivalent approach to defining this decomposition is given by Kawamata via the \emph{numerically fixed part} of a linear series~\cite{kawamataremarks}.
\begin{definition}
Suppose that \(\pi : X \to S\) is a nice family.  Then
\[
N_\sigma(D;X/S) = \lim_{\epsilon \to 0} \left( \inf {D^\prime : D^\prime \equiv_\pi D+\epsilon A, D^\prime \geq 0} \right)
\]
where the infimum of divisors is defined coefficient-wise.
\end{definition}

In the non-relative setting, the divisorial Zariski decomposition is defined for any pseudoeffective class \(D\), but it lacks certain useful properties of the classical Zariski decomposition in dimension \(2\).  In particular, the failure of the positive part \(P\) to be nef can be problematic. It is often useful to try to construct a birational model \(f : W \to X\) for which \(P_\sigma(f^\ast D)\) is actually nef.  This suggests that higher-dimensional versions of Zariski decomposition should allow passage to a higher birational model. There are several possible definitions, among them the weak Zariski decomposition of Birkar.

\begin{definition}[\cite{birkar2009}]
Suppose that \(\pi : X \to S\) is a nice family and that \(D\) is a pseudoeffective divisor on \(X\).  We say that \(D\) admits a \emph{weak Zariski decomposition over $S$} if there exists a birational map \(f : Y \to X\) and a decomposition \(f^\ast D = P +N \), where \(P\) is \(\pi\)-nef and \(N\) is effective.
\end{definition}
This condition is fairly unrestrictive in that it does not impose any analog of the negative-definiteness required in the two-dimensional setting.  Nevertheless, there exist pseudoeffective \(\R\)-divisors on smooth threefolds which do not admit a weak Zariski decomposition~\cite{jdlbminus}.  Corollary~\ref{nowzd} asserts that the divisor \(\tilde{D}\) provides another such example.  Indeed, \(\tilde{D}\) admits no Zariski decomposition in a still stronger sense: even after pulling back to a higher model, it cannot be decomposed as the sum of an effective divisor and a relatively movable divisor.  The example is qualitatively rather different from that of~\cite{jdlbminus}: there, a certain pseudoeffective divisor \(D_\lambda\) has negative intersection with infinitely many curves; here, there is a single curve on which \(D\) is negative, but the multiplicity of \(D\) along this curve is infinite.

\section{Main example}
\label{examplesect}

The claimed pathology follows from a few calculations on an example that has been studied by Kawamata and Reid.  Let \(\pi : X \to S\) be the versal deformation space of a fiber of Kodaira type \(I_2\).  The base \(S\) is smooth, \(2\)-dimensional germ.  The fiber over the central point \(0 \in S\) is consists of two smooth rational curves \(C_1\) and \(C_2\), meeting transversely at two points \(p_1\) and \(p_2\).  Let \(C = \pi^{-1}(0)\) be the union of these two curves.

There are two divisors \(\Gamma_1, \Gamma_2 \subset S\) corresponding to the smoothings of the two nodes of \(C\).  The fiber of \(\pi\) over a general point of \(\Gamma_i\) is a nodal rational curve, while the fiber over a general point of \(S\) is a smooth curve of genus \(1\).  

\begin{figure}[htb]
\centering
\begin{tikzpicture}[scale=0.8]
\def\perang{70}; \def\perlen{2};
\coordinate (si) at (\perlen,0); 
\coordinate (up) at ({\perlen*cos(\perang)},{\perlen*sin(\perang)});
\coordinate (ve) at (0,3);

\draw [thick] ($-1*(up)$) -- (up);
\draw [thick] ($-1*(si)$) -- (si);

\draw [ultra thin] ($(up)-(si)$) -- ($(up)+(si)$);
\draw [ultra thin] ($-1*(up)-(si)$) -- ($-1*(up)+(si)$);
\draw [ultra thin] ($-1*(si)-(up)$) -- ($-1*(si)+(up)$);
\draw [ultra thin] ($(si)-(up)$) -- ($(si)+(up)$);

\node [right] (G1) at (si) {$\Gamma_1$};
\node [below] (G2) at ($-1*(up)$) {$\Gamma_2$};

\begin{scope}[shift=(ve)]
\draw (-1/4,0) .. controls (5/12,2/3) and (5/12,4/3).. (-1/4,2);
\draw (1/4,0) .. controls (-5/12,2/3) and (-5/12,4/3).. (1/4,2);
\node [left] (C1) at (-0.1,3/2) {$C_1$};
\node [right] (C2) at (0.1,3/2) {$C_2$};
\end{scope}
\draw [densely dotted] (0,0) -- (ve);

\begin{scope}[shift=($-1*(up)+4/3*(ve)$)]
\draw (-1/4,0) .. controls (3/2,2) and (-3/2,2) .. (1/4,0);
\end{scope}
\draw [densely dotted] ($-1*(up)$) -- ($-1*(up)+4/3*(ve)$);

\begin{scope}[shift=($-1*(si)+(ve)$)]
\draw (-1/4,3/2) .. controls (3/2,-1/2) and (-3/2,-1/2) .. (1/4,3/2);
\end{scope}
\draw [densely dotted] ($-1*(si)$) -- ($-1*(si)+0.95*(ve)$);

\begin{scope}[shift=($1/2*(si)+1/2*(up)+3/4*(ve)$)]
\draw (0,0) .. controls (3/4,3) and (1/4,-3/2) .. (1,3/2);
\end{scope}
\draw [densely dotted] ($1/2*(si)+1/2*(up)+(1/2,0)$) -- ($1/2*(si)+1/2*(up)+0.85*(ve)+(1/2,0)$);
\end{tikzpicture}
\caption{The family \(\pi : X \to S\)}
\end{figure}
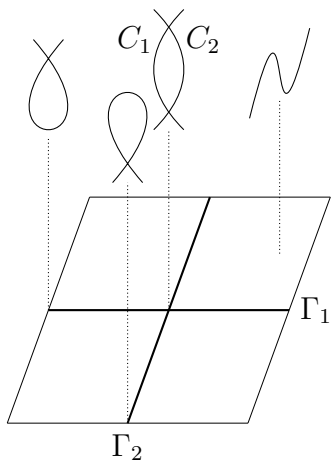

\begin{lemma}
The normal bundle \(N_{C_i/X}\) is isomorphic to \(\cO_{\P^1}(-1) \oplus \cO_{\P^1}(-1)\). 
\end{lemma}
\begin{proof}
Suppose for simplicity that \(i = 1\).  There is an exact sequence
\[
\xymatrix{
0 \ar[r] & N_{C_1/X} \ar[r] & (N_{C/X})\vert_{C_1} \ar[r] & T_{C_2,p_1} \oplus T_{C_2,p_2} \ar[r] & 0
}
\]
with the property that a first-order deformation, determined by a section \(s \in H^0(C,N_{C/X})\) smooths the node at \(p_i\) if and only if \(s\) has nonzero image \(T_{C_2,p_i}\)~\cite[Lemma 2.6]{graberharrisstarr}.   The sheaf in the middle is the trivial \(\cO_C \oplus \cO_C\). In one direction \(p_1\) is smoothed, and in another \(p_2\) is, so the map sends \((1,0)\) to \((1,0)\) and \((0,1)\) to \((0,1)\) with respect to the direct sum decompositions. It follows that the kernel is \(\cO_{\P^1}(-1) \oplus \cO_{\P^1}(-1)\).
\end{proof}

\begin{lemma}
\label{flopandiso}
There exists a flop \(\tau : X \rat X^+/S\) with flopping curve \(C_1\).  Let \(C_1^+ \subset X^+\) be the flopped curve, and \(C_2^\prime \subset X^+\) be the strict transform of \(C_2\). There exists an isomorphism \(\sigma : X^+ \to X/S\) which sends \(C_1^+\) to \(C_1\) and \(C_2^\prime\) to \(C_2\).  Furthermore, there exists an automorphism \(\imath : X \to X/S\) which exchanges the two curves \(C_1\) and \(C_2\).
\end{lemma}
\begin{proof}
The arguments here are due to Kawamata~\cite[Example 3.8(2)]{kawamatacy}.  We make some aspects of the proof explicit by working with local defining equations given by Reid~\cite{reid}.  In what follows, we use the notation \(\bar{\cdot}\) to denote objects on a family \(\bar{\pi} : \bar{X} \to \bar{S}\) over an affine base, while objects with no bar will be the restrictions to a certain germ.

Let \(\bar{S} = \A^2\), with coordinates \(t_1\) and \(t_2\).  Fix two distinct complex numbers \(a_1\) and \(a_2\) and define \(\bar{X}_0 \subset (\A^1 \times \A^1) \times \bar{S}\) by the equation
\[
x_1^2 = ((x_2-a_1)^2 - t_1)((x_2-a_2)^2-t_2).
\]

The closure \(\bar{X} \subset (\P^1 \times \P^1) \times \bar{S}\) is smooth, and the second projection \(\bar{\pi} : \bar{X} \to \bar{S}\) is proper.  The fiber of \(\bar{\pi}\) over a general point \((t_1,t_2)\) is a smooth curve of genus \(1\).   If exactly one of \(t_1\) and \(t_2\) is zero, the fiber is nodal, while if \(t_1 = t_2 = 0\), the fiber is given by \(x_1^2 = (x_2-a_1)^2(x_2-a_2)^2\).  This central fiber has two components, the rational curves \(C_1\) defined by \(x_1 = -(x_2-a_1)(x_2-a_2)\) and \(C_2\) defined by \(x_1 = (x_2-a_1)(x_2-a_2)\).  The restriction of \(\bar{\pi} : \bar{X} \to \bar{S}\) to the germ at \((0,0) \in \bar{S}\) is the versal deformation space \(\pi : X \to S\) considered above.  The involution \(\imath : \bar{X} \to \bar{X}/\bar{S}\) defined by \(\imath(x_1,x_2) =  (-x_1,x_2)\) exchanges the two components of the central fiber.

There is a section \(\bar{\sigma} : \bar{S} \to \bar{X}\) given by
\begin{align*}
x_2(t_1,t_2) &= \frac{a_1+a_2}{2} - \frac{t_1-t_2}{2(a_1-a_2)}, \\
x_1(t_1,t_2) &= (x_2(t_1,t_2)-a_1)^2 - t_1.
\end{align*}
This has \(\bar{\sigma}(0,0) = \left( \frac{(a_2-a_1)^2}{4}, \frac{1}{2}(a_1+a_2) \right)\), which lies on \(C_1\) and is disjoint from \(C_2\).  

Let \(\bar{\Sigma}_1\) be the divisor \(\sigma(\bar{S})\).  Since \(\bar{\Sigma}_1 \cdot C_1 = 1\) and \(\bar{\Sigma}_1 \cdot C_2 = 0\), the curves \(C_1\) and \(C_2\) have distinct classes in \(N_1(\bar{X}/\bar{S})\).  Since all other fibers of \(\bar{\pi}\) are irreducible, it must be that \(N^1(\bar{X}/\bar{S})\) is has dimension \(2\).  The divisor \(2 \imath_\ast(\bar{\Sigma}_1) - \bar{\Sigma}_1\) has positive degree on general fibers, and so is \(\bar{\pi}\)-big.  Since \(\bar{S}\) is affine, there is an effective divisor \(\bar{\Delta}\) representing this class.  For sufficiently small \(\epsilon\), the pair \((\bar{X},\epsilon \bar{\Delta})\) is klt.  Since \(\bar{\Delta} \cdot C_1 < 0\), there exists a \((K_{\bar{X}/\bar{S}}+\epsilon \bar{\Delta})\)-flip \(\tau : \bar{X} \rat \bar{X}^+\), which is a \(K_{\bar{X}/\bar{S}}\)-flop. The map \(\bar{\pi}^{+} : \bar{X}^+ \to \bar{S}\) is a minimal model of \(\bar{X}^+\).  The strict transform of \(\bar{\Sigma}_1\) on \(\bar{X}^+\) is smooth, contains the curve \(C_1^+\), and satisfies \(\tau_\ast \bar{\Sigma}_1 \cdot C_2^\prime = 2\).  

Since \(\pi : X \to S\) is a versal deformation space and \(\pi^+ : X^+ \to S\) has the same local structure, there exists an isomorphism \(\beta : X^+ \to X\) over \(S\).  However, this map might not be defined over the identity map on \(S\). The divisor \(\Sigma_2 = \beta_\ast(\tau_\ast(\Sigma_1))\) is a smooth divisor on \(X\), containing \(C_1\), and meeting \(C_2\) at two points.  There is a translation on the smooth fibers of \(\pi\) sending \(\Sigma_1\) to \(\Sigma_2\), which defines a birational automorphism \(\gamma : X \rat X\) over the identity on \(S\). The map \(\pi \circ \gamma : X \to S\) must be isomorphic to some minimal model of \(X\) over \(S\), and indeed must be to isomorphic to \(\pi^+ : X^+ \to S\) since the strict transforms of \(\Sigma_1\) under \(\gamma\) and \(\tau\) have the same numerical classes. It follows that there exists an isomorphism \(\sigma : X^+ \to X\) over the identity of \(S\).  Replacing \(\sigma\) with \(\sigma \circ \imath\) if necessary, we may assume that \(\sigma(C_1^+) = C_1\) and \(\sigma(C_2^\prime) = C_2\), as required.
\end{proof}

Each of the maps \(\sigma \circ \tau\) and \(\imath\) is a birational involution of \(X\) over \(S\), but we will soon see that the composition \(\phi = (\sigma \circ \tau) \circ \imath\) is of infinite order.  Since \(\imath(C_2) = C_1\), the effect of repeatedly applying \(\phi\) is to flop \(C_1\), then \(C_2\), then \(C_1\) again, and so on. We will denote by \(\phi_\ast D\) the strict transform of a divisor \(D\) under a birational map \(\phi\), and use the same notation for the induced map on numerical groups when confusion seems unlikely.

\begin{figure}[htb]
\centering
\begin{tikzpicture}[scale=0.6]
\def\perang{75}; \def\perlen{2};
\coordinate (si) at (\perlen,0); 
\coordinate (up) at ({\perlen*cos(\perang)},{\perlen*sin(\perang)});
\newcommand{\basespace}{%
\draw [thick] ($(sh)-(up)$) -- ($(sh)+(up)$);
\draw [thick] ($(sh)-(si)$) -- ($(sh)+(si)$);
\draw [ultra thin] ($(sh)+(up)-(si)$) -- ($(sh)+(up)+(si)$);
\draw [ultra thin] ($(sh)-(up)-(si)$) -- ($(sh)-(up)+(si)$);
\draw [ultra thin] ($(sh)-(si)-(up)$) -- ($(sh)-(si)+(up)$);
\draw [ultra thin] ($(sh)+(si)-(up)$) -- ($(sh)+(si)+(up)$);
}
\coordinate (x0center) at (0,0);
\coordinate (x1center) at (16,0);
\coordinate (rescenter) at (8,3);

\coordinate (sh) at (x0center);
\basespace
\coordinate (sh) at (x1center);
\basespace
\coordinate (sh) at (rescenter);
\basespace

\draw[->, shorten >=2cm, shorten <=2cm] (rescenter) -- (x0center) node [midway, above] {$f$};
\draw[->, shorten >=2cm, shorten <=2cm] (rescenter) -- (x1center) node [midway, above] {$g$};
\draw[->, dashed, shorten >=2cm, shorten <=2cm] (x0center) -- (x1center) node [midway, below] {$\tau$};

\newcommand{\centralfiber}{%
\coordinate (ve) at (0,3);
\draw [very thick] ($(sh)+(-1/4,0)+(ve)$) .. controls ($(sh)+(5/12,2/3)+(ve)$) and ($(sh)+(5/12,4/3)+(ve)$).. ($(sh)+(-1/4,2)+(ve)$);
\draw [very thick] ($(sh)+(1/4,0)+(ve)$) .. controls ($(sh)+(-5/12,2/3)+(ve)$) and ($(sh)+(-5/12,4/3)+(ve)$).. ($(sh)+(1/4,2)+(ve)$);
\draw [densely dotted] (sh) -- ($(sh)+(ve)$);
}
\coordinate (sh) at ($(x0center)$);
\centralfiber
\node [left] (C1) at ($(sh)+(ve)+(-0.1,3/2)$) {$C_1$};
\node [right] (C2) at ($(sh)+(ve)+(0.1,3/2)$) {$C_2$};
\coordinate (sh) at ($(x1center)$);
\centralfiber
\node [left] (C1) at ($(sh)+(ve)+(-0.1,3/2)$) {$C_1^+$};
\node [right] (C2) at ($(sh)+(ve)+(0.1,3/2)$) {$C_2^\prime$};

\begin{scope}[shift=($(rescenter)+(ve)$)]
\begin{scope}[rotate=45]
\draw[step=0.625cm] (0,0) grid (5/2,5/2);
\draw[very thick] (0,0) -- (5/2,0);
\draw[very thick] (0,0) -- (0,5/2);
\node [draw,shape=circle,fill=black,inner sep=0pt,minimum width=1mm] (tt) at (5/8,5/8) {};
\node [draw,shape=circle,fill=black,inner sep=0pt,minimum width=1mm] (tt) at (15/8,15/8) {};
\draw [very thick] plot [smooth,tension=0.8] coordinates {(-5/8,5/8) (5/8,5/8) (25/16,15/16) (15/8,15/8) (15/8,25/8)};
\end{scope}
\node [left] (C1r) at (135:5/4) {$\bar{C}_1$};
\node [right] (C1s) at (45:5/4) {$\bar{C}_1^+$};
\node [left] (C2p) at (90:15/4) {$\bar{C}_2^\prime$};
\end{scope}
\draw [densely dotted, shorten >=1mm] (rescenter) -- ($(rescenter)+(ve)$);

\end{tikzpicture}
\caption{Resolution of the flop $\tau$}
\end{figure}
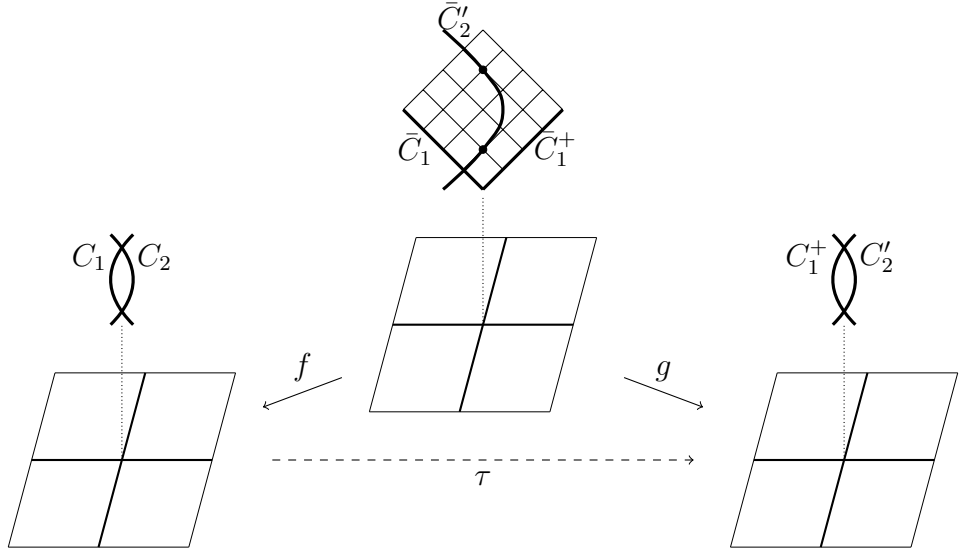

To an effective divisor \(D\) on \(X\), associate the \(4\)-tuples
\begin{align*}
v_D &= (D \cdot C_1,D \cdot C_2,\mult_{C_1}(D),\mult_{C_2}(D)), \\ 
\sigma_D &= (D \cdot C_1,D \cdot C_2,\sigma_{C_1}(D;X/S),\sigma_{C_2}(D;X/S)).
\end{align*}
\begin{lemma}
\label{multmatrix}
Suppose that \(D\) is a divisor on \(X\), and let \(\tilde{D}\) denote the strict transform of \(D\) under the flop \(\tau : X \rat X^+\).  Then
\begin{enumerate}
\item \(\tilde{D} \cdot C_1^+ = -D \cdot C_1\),
\item \(\tilde{D} \cdot C_2^\prime = D \cdot C_2 + 2( D \cdot C_1)\),
\item \(\mult_{C_1^+}(\tilde{D}) = \mult_{C_1}(D) + D \cdot C_1\),
\item \(\mult_{C_2^\prime}(\tilde{D}) = \mult_{C_2}(D)\).
\end{enumerate}
In matrix form, we have \(v_{\phi_{\ast}{D}} = M v_D\) and \(\sigma_{\phi_{\ast}D} = M \sigma_D\) where
\[
M = \left( \begin{array}{rrrr} 2 & 1 & 0 & 0 \\ -1 & 0 & 0 & 0 \\ 0 & 0 & 0 & 1  \\ 1 & 0 & 1 & 0 
\end{array} \right)
\]
\end{lemma}
\begin{proof}
Let \(W\) be the graph of the flop \(\tau\):
\[
\xymatrix{
& W \ar[dl]_f \ar[dr]^g \\ X \ar@{-->}^{\tau}[rr] && X^+ }
\]
Since \(\tau\) is the flop of a rational curve with normal bundle \(\cO_{\P^1}(-1) \oplus \cO_{\P^1}(-1)\), there is a single \(f\)-exceptional divisor \(E\) on \(W\), which is isomorphic to \(\P^1 \times \P^1\) and has normal bundle of bidegree \((-1,-1)\).  Let \(\bar{C}_1\) be a ruling of \(E\) contracted by \(g\), so that \(f\) sends \(\bar{C}_1\) isomorphically to \(C_1\).  Similarly, let \(\bar{C}_1^+\) be a a ruling of \(E\) contracted by \(f\), so that \(g\) maps \(\bar{C}_1^+\) isomorphically onto \(C_1^+\).  Lastly, let \(\bar{C}_2^\prime\) be the strict transform of \(C_2\) on \(W\), a curve which meets \(E\) transversely at \(2\) points. Then write
\[
f^\ast D + aE = g^\ast \tilde{D},
\]
for some constant \(a\).  Taking the intersection of both sides with \(\bar{C}_1\) yields \(D \cdot C_1 + a(E \cdot \bar{C}_1) = 0\).  Since \(E \cdot \bar{C}_1 = -1\), we obtain \(a = D \cdot C_1\).  Intersecting with \(\bar{C}_1^+\), we have \(-a = \tilde{D} \cdot C_1^+\).  Similarly, intersecting with \(\bar{C}_2^\prime\), we have \(D \cdot C_2 + a (E \cdot \bar{C}_2^\prime) = \tilde{D} \cdot C_2^\prime\), and since \(E \cdot \bar{C}_2^\prime = 2\), we have (2).  It is clear that \(\mult_{C_2^\prime}(\tilde{D}) = \mult_{C_2}(D)\), since \(\tau\) is an isomorphism at the generic point of \(C_2\).  Finally,
\[
\mult_{C_1^+}(\tilde{D}) = \mult_E(g^\ast \tilde{D}) = \mult_E(f^\ast D) + a = \mult_{C_1}(D) + a.
\]

These calculations immediately yield \(v_{\phi_{\ast}{D}} = M v_D\), since the second map \(\imath\) exchanges the two curves \(C_1\) and \(C_2\).  Write \(D_m\) for a general divisor linearly equivalent to \(mD\), and then 
\begin{align*}
\sigma_{C_1}(\phi_\ast D) &= \lim_{m \to \infty} \frac{1}{m} \mult_{C_1}(\phi_\ast D_m) = \lim_{m \to \infty} \frac{1}{m} \left( \mult_{C_1} D_m + D_m \cdot C_1 \right) \\
&= \left( \lim_{m \to \infty} \frac{1}{m} \mult_{C_1} D_m \right) + D \cdot C_1 =  \sigma_{C_1}(D) + a. \qedhere
\end{align*}
\end{proof}

We are now in position to make the main computation.
\begin{theorem}
\label{computeiterates}
Let \(\pi : X \to S\) be the versal deformation space of a singular fiber of Kodaira type \(I_2\), and let \(C_1\) be a component of the central fiber.  Suppose that \(D\) is a divisor on the boundary of the cone \(\Effb(X/S)\).  Then \(\sigma_{C_1}(D;X/S) = \infty\).
\end{theorem}
\begin{proof}

Fix a \(\pi\)-ample effective \(\Q\)-divisor \(H = H_0\) on \(X\) with \(H \cdot C_1 = H \cdot C_2 = 1\) and \(\mult_{C_i}(H) = 0\). Let \(H_n = \phi_\ast^n(H)\) be the strict transform of \(H\) on \(X\) under \(n\) applications of \(\phi\).  Using the Jordan decomposition of \(M\), which has a \(3 \times 3\) block associated to the eigenvalue \(1\), we compute \(\sigma_{H_n} = (2n+1,-2n+1,n(n-1)/2,n(n+1)/2)\):
\[
\begin{array}{r|r|r|r|r}
n & H_n \cdot C_1 & H_n \cdot C_2 & \mult_{C_1} H_n & \mult_{C_2} H_n \\\hline 0 & 1 & 1 & 0 & 0 \\ 1 & 3 & -1 & 0 & 1 \\ 2 & 5 & -3 & 1 & 3 \\ 3 & 7 & -5 & 3 & 6 \\ && \cdots && \\ n & 2n+1 & -2n+1 & \frac{n(n-1)}{2} & \frac{n(n+1)}{2}
\end{array}
\]

The key feature of the example is that while \(H_n \cdot C_1\) grows linearly in \(n\), the multiplicity \(\mult_{C_1}(H_n)\) grows quadratically. Let \(D\) be the divisor class on the boundary of \(\Effb(X/S)\) with \(D \cdot C_1 = 1\) and \(D \cdot C_2 = -1\).  Since \(C_1\) and \(C_2\) span \(N_1(X/S)\), we see that
\[
H_n \equiv_\pi (2n)D + H_0.
\]

It follows that \(\frac{1}{2n} H_n \equiv_\pi D + \frac{1}{2n} H_0\) is a sequence of divisors converging to \(D\), whose multiplicities along the curves is known.  By Definition~\ref{sigmadef}, we compute
\[
\sigma_{C_1}(D;X/S) = \lim_{n \to \infty} \mult_{C_1} (D + \frac{1}{2n} H_0) = \lim_{n \to \infty} \frac{1}{2n} \mult_{C_1} H_n = \lim_{n \to \infty} \frac{n-1}{4} = \infty. \qedhere
\]
\end{proof}

Note that \(\codim \pi(C_1) = 2\), so there is no contradiction with Theorem~\ref{finiteconds}(3).

\begin{corollary}
\label{nozd}
If \(f : W \to X\) is the blow-up along \(C_1\) with exceptional divisor \(E\), then \(\tilde{D} = f^\ast D\) has \(\sigma_E(\tilde{D}; W/S) = \infty\) and \(N_\sigma(\tilde{D};W/S)\) contains the divisor \(E\) with infinite coefficient.  In particular, there does not exist a birational model \(g : Z \to W\) for which \(g^\ast \tilde{D}\) admits a decomposition \(g^\ast \tilde{D} = P+N\) with \(P\) a \(g \circ (f \circ \pi)\)-movable divisor and \(N\) effective.
\end{corollary}
\begin{proof}
By Theorem~\ref{sigmaproperties}(2), if \(f : W \to X\) is the blow-up along \(C_1\), with exceptional divisor \(E\), we have \(\sigma_E(f^\ast D;W/S) = \infty\). Now, suppose that \(g : Y \to W\) is any birational map, and that \(g^\ast f^\ast D = P + N\), where \(P\) is a \((g \circ f \circ \pi)\)-movable divisor and \(N\) is effective.  Let \(\tilde{E}\) denote the strict transform of \(E\) on \(Y\). Then
\[
\sigma_{\tilde{E}}(g^\ast f^\ast D;Y/S) \leq \sigma_{\tilde{E}}(P;Y/S) + \sigma_{\tilde{E}}(N;Y/S) = \sigma_{\tilde{E}}(N;Y/S).
\]
The last of these is finite since \(N\) is effective, while the first is infinite, a contradiction. This completes the proof.
\end{proof}

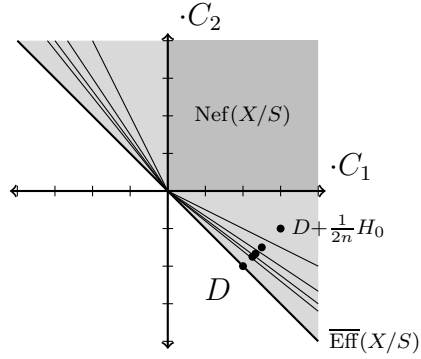
\begin{figure}[htb]
\centering
\begin{tikzpicture}[dot/.style={circle,fill=black,minimum size=3pt,inner sep=0pt,outer sep=-1pt}]

\fill[gray!30] (-2,2) -- (2,2) -- (2,-2) -- cycle;
\fill[gray!50] (0,0) -- (2,0) -- (2,2) -- (0,2) -- cycle;

\draw [<->,thick] (-2.1,0) -- (2.1,0);
\draw [<->,thick] (0,-2.1) -- (0,2.1);
\foreach \xx[evaluate={\x=\xx/2}] in {-4,-3,...,4} {
\draw (\x,2pt) -- (\x,-2pt);
\draw (2pt,\x) -- (-2pt,\x);
}

\foreach \x in {0,...,4} {
\draw (0,0)--($ (2,{-2*\x/(\x+1)}) $);
\draw (0,0)--($ ({-2*\x/(\x+1)},2) $);
}
\draw [thick] (0,0)--(2,-2);
\draw [thick] (0,0)--(-2,2);

\node [above right] (C1) at (2,0) {$\cdot C_1$};
\node [above right] (C2) at (0,2) {$\cdot C_2$};
\node [right] (EF) at (2,-2) {$\scriptstyle \Effb(X/S)$};
\node (NF) at (1,1) {$\scriptstyle \Nef(X/S)$};
\node [dot,label=below left:$D$] (D) at (1,-1) {};

\node [right] (EF) at (3/2,-1/2) {$\scriptstyle D+\frac{1}{2n} H_0$};

\foreach \x in {1,...,4} {
\node [dot] (D) at ($ ({1+1/(2*\x)},{-1+1/(2*\x)}) $) {};
}
\end{tikzpicture}
\caption{Chambers in $N^1(X/S)$}
\end{figure}

\begin{remark}
\label{affinebase}
The construction of the automorphism \(\beta : X^+ \to X/S\) in Lemma~\ref{flopandiso} (from \cite{kawamatacy}) is only over a surface germ \(S\), for it relies on the fact that \(\pi : X \to S\) is a versal deformation space.  However, the local analytic results of Theorem~\ref{computeiterates} and Corollary~\ref{nozd} imply that the same pathological behavior occurs even when the base \(S\) is an affine surface.  We have seen that there is a projective family \(\bar{\pi} : \bar{X} \to \bar{S}\) where \(\bar{S}\) is an affine surface, such that the restriction of \(\bar{\pi}\) to the germ at a point \(0 \in \bar{S}\) coincides with the map \(\pi : X \to S\).

If \(\bar{G}\) is a \(\bar{\pi}\)-big divisor, with restriction \(G\) to the germ, then \(\sigma_{C_1}(\bar{G};\bar{X}/\bar{S}) \geq \sigma_{C_1}(G;X/S)\): indeed, if \(\bar{G}^\prime\) is an effective divisor on \(\bar{X}\) which is \(\bar{\pi}\)-numerically equivalent to \(\bar{G}\), its restriction to the central germ is an effective divisor on \(X\) which is \(\pi\)-numerically equivalent to \(G\).  Thus the infimum defining \(\sigma_{C_1}(\bar{G};\bar{X}/\bar{S})\) in is taken over a subset of the infimum defining \(\sigma_{C_1}(G;X/S)\) in Definition~\ref{sigmadef}, giving the claimed inequality. It follows that in the limit at the pseudoeffective boundary, \(\sigma_{\bar{C_1}}(\bar{D};\bar{X}/\bar{S}) \geq \sigma_{C_1}(D;X/S)\), and it must be that \(\sigma_{\bar{C_1}}(\bar{D};\bar{X}/\bar{S})\) is infinite as well.  The claims about Zariski decomposition follow as before.
\end{remark}

\section{A general set-up}
\label{generalsetup}

The key feature that made possible the computation of the preceding example is that if the four numbers \(D \cdot C_i\) and \(\mult_{C_i}(D)\) are all known, then the same four invariants can be computed for the strict transform of \(D\) under \(\phi\) using Lemma~\ref{multmatrix}.  In this section, we give an explanation for this, and describe how to make analogous computations in a more general setting.

Suppose that \(\phi : X \rat X\) is a pseudoautomorphism over \(S\), i.e.\ a birational map for which neither \(\phi\) nor \(\phi^{-1}\) contracts any divisors.  We will say that a birational morphism \(f : Y \to X\) from a normal \(\Q\)-factorial variety \(Y\) is a \emph{small lift} of \(\phi\) if the induced map \(\psi : Y \rat Y\) is also a pseudoautomorphism.
\[
\xymatrix{
Y \ar[d]_f \ar@{-->}[r] \ar@{-->}[r]_\psi & Y \ar[d]_f \\ X \ar@{-->}[r]_\phi & X }
\]
Observe that if \(f : Y \to X\) is a small lift, then the map \(\psi\) must permute the exceptional divisors of \(f\).  

\begin{example}
Suppose that \(\phi : X \rat X\) is a pseudoautomorphism and \(x\) is a point not contained in \(\indet \phi\).  The blow-up \(f : \Bl_x X \to X\) is a small lift of \(\phi\) if and only if \(x\) is a fixed point of \(\phi\).  If \(x\) is not a fixed point, then the induced map \(\psi : Y \rat Y\) contracts the exceptional divisor \(E\), while if \(x\) is fixed, then \(\psi\vert_E : E \to E\) is an automorphism.
\end{example}

The more interesting examples are those in which \(f\) contracts a divisor lying over \(\indet \phi\).

\begin{example}
\label{mainexample}
Next we construct a small lift of the map \(\phi : X \rat X/S\) from Section~\ref{examplesect}.  Let \(f : W \to X\) be the blow-up along \(C_1\) as before, with exceptional divisor \(E_1\), and let \(h : Y \to W\) be the blow-up along \(\bar{C}_2^\prime\), with exceptional divisor \(E_2\). The two exceptional divisors \(E_1\) and \(E_2\) are swapped by the induced map \(\psi : Y \rat Y\), and \(h \circ f\) is a small lift.
\[
\xymatrix{
Y \ar@{-->}[rr]^\psi \ar[dd]_{h \circ f} \ar[dr]^{h} && Y \ar[dd]^{h \circ f} \\ 
& W \ar[dl]_f \ar[dr]^g & \\ 
X \ar@{-->}[rr]^{\phi}  && X 
}
\]
The curves \(C_1\) and \(C_2\) could have been blown up in the opposite order, yielding a different small lift \(f^\prime : Y^\prime \to X\).  This makes no real difference: the threefolds \(Y\) and \(Y^\prime\) differ only by flops, and strict transform induces an identification \(N^1(Y) \xrightarrow{\sim} N^1(Y^\prime)\) with respect to which the maps \(\psi_\ast\) and \(\psi_\ast^\prime\) coincide.
\end{example}

If \(f : Y \to X\) is a small lift, it follows from the negativity lemma~\cite[Lemma 3.6.2]{bchm} that there is a decomposition \(N^1(Y) = f^\ast N^1(X) \oplus V_E\), where \(V_E = \bigoplus_i \R \cdot [E_i]\).  If \(D\) is a divisor class on \(X\), it is not necessarily true that \(f^\ast \phi_\ast D = \psi_\ast f^\ast D\). However, the difference \(f^\ast \phi_\ast D - \psi_\ast f^\ast D \) is an \(f\)-exceptional divisor, since
\[
f_\ast(f^\ast \phi_\ast D - \psi_\ast f^\ast D) = \phi_\ast D - f_\ast \psi_\ast f^\ast D = \phi_\ast D - \phi_\ast f_\ast  f^\ast D = \phi_\ast D - \phi_\ast D = 0.
\]
Define \(K : N^1(X) \to V_E\) by \(K = f^\ast \phi_\ast-\psi_\ast f^\ast \).  The next lemma characterizes the action of the strict transform \(\psi_\ast : N^1(Y) \to N^1(Y)\) with respect to this decomposition.

\begin{lemma}
\label{liftmap}
Suppose that \(f : Y \to X\) is a small lift of a pseudoautomorphism \(\phi : X \rat X\).  With respect to the decomposition \(N^1(Y) \cong f^\ast N^1(X) \oplus V_E\), \(\psi_\ast\) is given in block form as
\[
\psi_\ast = \left( \begin{array}{c|c} \phi_\ast & 0 \\\hline
-K & P \end{array} \right),
\]
where \(P\) is the permutation matrix for the action of \(\psi_\ast\) on the \(E_i\).  The eigenvalues of \(\psi_\ast\) are the union of those of \(\phi_\ast\) and those of \(P\), which are roots of unity. Its eigenvectors are
\begin{enumerate}
\item \(f^\ast v_i - (\lambda I - P)^{-1} Kv_i\), where \(v_i\) are the eigenvectors of \(\phi_\ast\), with eigenvalues \(\lambda_i\);
\item \(E_i\), the exceptional divisors of \(f\), with eigenvalues that are roots of unity.
\end{enumerate}
\end{lemma}
\begin{proof}
For a divisor \(D\) on \(X\), \(\psi_\ast f^\ast D = f^\ast \phi_\ast D - KD\), while the exceptional divisors \(E_i\) are simply permuted by \(\psi\); this gives the block form of the map. The eigenvectors follow from elementary linear algebra.
\end{proof}

A rational map \(\phi : X \to Y\) is said to be \(D\)-non-negative for an \(\R\)-divisor \(D\) if on some common resolution \(f : W \to X\), \(g : W \to Y\), we have \(f^\ast D +E= g^\ast(\phi_\ast D)\), where \(E\) is an effective \(g\)-exceptional divisor. If \(\phi : X \rat X\) is a pseudoautomorphism with a small lift \(f\), then we may consider a resolution of the form
\[
\xymatrix{
& W \ar[dl]_p \ar[dr]^q \\ Y \ar[d]_f  \ar@{-->}[rr]_\psi && Y \ar[d]_f \\ X \ar@{-->}[rr]_\phi && X }
\]
If \(D\) is a divisor on \(X\) for which \(\phi\) is \(D\)-non-negative, then we have \(p^\ast f^\ast D +E= q^\ast f^\ast \phi_\ast D \) with \(E \geq 0\).  Pushing forward both sides by \(q\), this gives 
\begin{align*}
q_\ast p^\ast f^\ast D +q_\ast E&= f^\ast \phi_\ast D \\
\psi_\ast f^\ast D + E^\prime &= f^\ast \phi_\ast D,
\end{align*}
where \(E^\prime\) is an effective \(f\)-exceptional divisor.  In particular, \(KD = f^\ast \phi_\ast D- \psi_\ast f^\ast D  = E^\prime\) is effective.

Next we observe that if the divisorial Zariski decomposition \(P_\sigma(f^\ast D)\) is known for some divisor \(D\), the decomposition \(P_\sigma(f^\ast\phi_\ast D)\) can often be computed, using the strict transform under \(\psi : Y \rat Y\). For simplicity, we assume that \(\psi\) fixes each of the \(f\)-exceptional divisors \(E_i\); this can always be arranged by replacing \(\phi\) by a suitable iterate.  This assumption implies that the permutation matrix \(P\) is the identity, and that \(\psi_\ast (KD) = KD\) since \(KD\) is exceptional.

\begin{lemma}
\label{pullbackzd}
Suppose \(\phi : X \rat X\) is a pseudoautomorphism over \(S\), and that \(D\) is a class in \(N^1(X/S)\). Then \(N_\sigma(\phi_\ast D;X/S) = \phi_\ast N_\sigma(D;X/S)\).  If \(N_\sigma(D;X/S)\) is finite, then \(P_\sigma(\phi_\ast D;X/S) = \phi_\ast P_\sigma(D;X/S)\) as well.  If \(\phi\) is \(D\)-non-negative and \(N_\sigma(f^\ast D;Y/S)\) is finite, then \(P_\sigma(f^\ast \phi_\ast D;Y/S) = \psi_\ast P_\sigma(f^\ast D;Y/S)\).
\end{lemma}
\begin{proof}
Since \(\phi\) neither contracts nor extracts any divisors, for any prime divisor \(E\) we have \(\sigma_E(D;X/S) = \sigma_{\phi_\ast E}(\phi_\ast D;X/S)\). The claim for \(N_\sigma(\phi_\ast D;X/S)\) follows, and that for \(P_\sigma(\phi_\ast D;X/S)\) is immediate.

Now, by the \(D\)-non-negativity hypothesis on \(\phi\), \(KD\) is an effective exceptional divisor.  By ~\cite[Lemma 3.5.1]{nakayama}, if \(E\) is an effective exceptional divisor, we have \(N_\sigma(f^\ast D + E) = N_\sigma(f^\ast D) + E\).  This means that
\begin{align*}
N_\sigma(f^\ast \phi_\ast D) &= N_\sigma(\psi_\ast f^\ast D +KD)= N_\sigma(\psi_\ast (f^\ast D + KD)) = \psi_\ast N_\sigma(f^\ast D + KD) \\ &= \psi_\ast N_\sigma(f^\ast D) + \psi_\ast KD = \psi_\ast N_\sigma(f^\ast D) + KD.
\end{align*}
We have made use of the fact that \(E\) is effective by the non-negativity hypothesis on \(D\). It is now simple to compute the positive part of the decomposition:
\begin{align*}
P_\sigma(f^\ast \phi_\ast D) &= f^\ast \phi_\ast D - N_\sigma(f^\ast \phi_\ast D) = f^\ast \phi_\ast D - \psi_\ast N_\sigma(f^\ast D) - KD \\ 
&= \psi_\ast f^\ast D - N_\sigma(\psi_\ast f^\ast D) = P_\sigma(\psi_\ast f^\ast D) = \psi_\ast P_\sigma(f^\ast D). \qedhere
\end{align*}
\end{proof}

\begin{remark}
The example of Section~\ref{examplesect} can be interpreted as an instance of the calculations in this section.  A small lift of the map \(\phi\) is constructed in Example~\ref{mainexample}.  Let \(F_1,F_2\) be a basis for \(N^1(X/S)\) dual to \(C_1\) and \(C_2\). A basis for \(N^1(Y/S)\) is given by the four classes \((h \circ f)^\ast F_1\), \((h \circ f)^\ast F_2\), \(E_1\), and \(E_2\).  The vector \(v_D\) gives the coefficients for the class of the strict transform of \(D\) on \(Y\) with respect to this above basis. Lemma~\ref{multmatrix} is nothing more than the calculation of the induced map \(\psi_\ast\) of Lemma~\ref{liftmap}.  The final calculation in Theorem~\ref{computeiterates} can then be carried out as a repeated application of Lemma~\ref{pullbackzd}.
\end{remark}

Suppose now that \(S = \Spec \C\) and \(\phi : X \rat X\) is a pseudoautomorphism whose action on \(N^1(X)\) has a unique largest eigenvalue, greater than \(1\), and that \(f : Y \to X\) is a small lift of \(\phi\).  We are then able to compute the Zariski decomposition of the divisor \(f^\ast D_\phi\) using the above result.
\begin{corollary}
Let \(D_\phi\) be the dominant eigenvector of \(\phi_\ast : N^1(X) \to N^1(X)\), and \(D_\psi\) be the dominant eigenvector of \(\psi_\ast : N^1(Y) \to N^1(Y)\). Then \(P_\sigma(f^\ast D_\phi) = D_\psi\).
\end{corollary}
\begin{proof}
If \(D\) is any pseudoeffective divisor on \(X\), then for every \(n\) we have 
\[ 
P_\sigma (f^\ast ( \lambda^{-n} \phi_\ast^n D )) =  \lambda^{-n} \psi_\ast^n P_\sigma(f^\ast D).
\]
Take \(D = D_\phi + D_{\phi^{-1}}\), so that the above reduces to
\[
P_\sigma(f^\ast(D_\phi + \lambda^{-2n} D_{\phi^{-1}}))  = \lambda^{-n} \psi_\ast^n P_\sigma(f^\ast D).
\]
The left hand side converges to \(P_\sigma(f^\ast D_\phi)\) by Proposition~\ref{sigmaproperties}(1).  With a suitable choice of scaling, the right hand side converges to \(D_\psi\).
\end{proof}

\section{Acknowledgements}
I am grateful to James M\textsuperscript{c}Kernan for some useful questions and suggestions.  This material is based upon work supported by the National Science Foundation under agreement No.\ DMS-1128155. Any opinions, findings and conclusions or recommendations expressed in this material are those of the author and do not necessarily reflect the views of the National Science Foundation.

\bibliographystyle{amsplain}
\bibliography{zrefs}

\end{document}